\documentclass{article}
\usepackage[english]{babel}
\usepackage[utf8]{inputenc}
\usepackage{listings}
\usepackage{graphicx}
\usepackage{dsfont}
\usepackage{amsmath,amssymb}
\usepackage{xcolor}
\usepackage{cancel}
\usepackage{amsthm}

\author{Sergio López-Ureña}
\title{A Harten's Multiresolution Framework\\ for Subdivision Schemes}
\date{January 2019}

\newcommand{\R}{\mathds{R}}
\newcommand{\Z}{\mathds{Z}}

\newcommand{\ldZ}{\ell_2(\mathds{Z})}

\newcommand{\iZ}{{i\in\Z}}
\newcommand{\lra}{\longrightarrow}
\DeclareMathOperator*{\argmin}{argmin}
\DeclareMathOperator*{\spa}{span}
\newcommand{\cD}{\mathcal{D}}
\newcommand{\cC}{\mathcal{C}}
\newcommand{\cF}{\mathcal{F}}
\newcommand{\cR}{\mathcal{R}}
\newcommand{\cV}{\mathcal{V}}

\newcommand{\cW}{\mathcal{W}}
\newcommand{\dS}{\cR}
\newcommand{\nS}{P}

\newcommand{\espan}{\text{span}}

\theoremstyle{plain}
\newtheorem{thm}{Theorem}
\newtheorem{lem}[thm]{Lemma}
\newtheorem{teo}[thm]{Theorem}
\newtheorem{cor}[thm]{Corolary}

\newtheorem{defi}[thm]{Definition}

\theoremstyle{remark}

\begin{document}

\maketitle

\begin{abstract}
Harten's Multiresolution framework has been applied in different contexts, such as in the numerical simulation of PDE with conservation laws or in image compression, showing its flexibility to describe and manipulate the data in a multilevel fashion. Two basic operators form the basis of this theory: the decimation and the prediction. The decimation is chosen first and determines the type of data that is being manipulated. For instance, the data could be the point evaluations or the cell-averages of a function, which are the two classical environments. Next, the prediction is chosen, and it must be compatible with the decimation.

Subdivision schemes can be used as prediction operators, but sometimes they not fit into one of the two typical environments. In this paper we show how to invert this order so we can choose a prediction first and then define a compatible decimation from that prediction. Moreover, we also prove that any possible decimation can be obtained in this way.
\end{abstract}

\section{Introduction}

Subdivision schemes are a valuable technique for the refinement of data, very common in Computer Aided Geometric Design (CAGD) \cite{Dyn92}, that allows to generate curves (or surfaces or even manifolds) from an initial discrete data set, namely $v^0$. A subdivision scheme recursively generates more and more data sets $v^k$, $k\in\Z_+$, and produces a function from these data sets that it is used as the parametrization of a geometrical object, if the scheme is convergent.

Despite that the subdivision was conceived with a geometrical purpose in CAGD \cite{DLG90}, other applications have been adopted it because of its easy implementation and flexibility to reach special properties. We are interested in the presence of subdivision in multiresolution algorithms, with applications in image processing \cite{ADLT06,CM02}, optimization \cite{DL-UM16,L-UT-LDG-A-C18} and uncertainty quantification \cite{DL-U17}, among others.

Originally, Harten's Mulitiresolution Framework (HMR-F) \cite{Harten96} provided a set of tools that allows to define a consistent multi-scale structure for numerical methods for \emph{conservation laws}. Nevertheless, this theory is prepared for very general multi-scale scenarios and over the years it were found applications in other mathematical fields, for instance in the above mentioned applications \cite{ADLT06,CM02,DL-UM16,DL-U17,L-UT-LDG-A-C18}, where it was combined with subdivision schemes.

In HMR-F, two basic operations are present: \emph{decimation} and \emph{prediction}. In \cite{Harten96} a very detailed study is found where the decimation is a linear operator which is defined before the prediction. Then, a (possibly nonlinear) prediction operator is picked, which must be \emph{consistent} with the decimation previously chosen. This order (first decimation, then prediction) is crucial in conservation laws, because the decimation establishes if a cell-average or a point-value framework has been chosen, which indicates how the prediction should be designed.

In other applications, however, the prediction operators are subdivision schemes and they seem to be more relevant than the decimation operators. In fact, in \cite{DL-U17,DL-UM16,L-UT-LDG-A-C18}, the decimation does not appear in the implementation of the methods, despite being needed to describe the multiresolution structure.
The aim of this paper is to give theoretical support to this procedure. We will not only prove that decimation operators can be defined consistently from linear prediction operators (subdivision schemes), but also that every decimation operator chosen for a prediction operator can be derived from such prediction. Thus, a well-defined multiresolution setting is always guaranteed in those applications where the prediction is based on linear subdivision schemes.

The nonlinear subdivision schemes that have been applied in practical situations until now (see for example \cite{ADL11,ADLT06,AL05,ADS16,Aslam18,DL-U18,DL-UM16}) are usually compatible with either the point-value framework or the cell-average framework. Since our interest for the present work lays outside of these frameworks, we will focus only on the linear case. This case is still of relevance since there are linear subdivision schemes that do not fit in any of both common frameworks, such as the (exponential) B-Splines family of subdivision schemes \cite{CR11,Dyn92}. Our results provide compatible decimation and discretization operations for these situations.

The paper is organized as follows: In Section \ref{sec:harten}, we introduce the main concepts of the HMR-F which are necessary to understand the situation we would like to study. In Section \ref{sec:construction} the construction of decimation and discretization operators from linear prediction and reconstruction operators is performed, and consequently,  by Theorem \ref{teo:hierar_define} a complete multiresolution setting is obtained. Morevoer, in Theorem \ref{teo:equivalence} we prove that every linear multiresolution framework can be built from the prediction and reconstruction operators. We apply our results to a classical family of subdivision schemes in Section \ref{sec:example}.

\section{Harten's Multiresolution Framework} \label{sec:harten}

In signal processing, data is often a discrete representation of a function. For example, consider a digital photo of a landscape, which is actually a set of pixels. Once the picture is taken, we can apply some treatments to it, such as denoising, improving the resolution, object classification, etc. All these operations belong to the signal processing field, and in particular, to image processing.

When a signal is processed, the discrete data is manipulated taking into account the underlying continuous nature that it is approximating. In the previous example, to denoise (that is, to remove noise present in a picture) we may assume that the color throughout the picture comes from a piecewise smooth function, so isolated sudden changes in the color may be identify as noise.

Harten's Multiresolution Framework (HMR-F) can be used to handle properly this kind of signals in a multi-scale fashion. A complete description of the HMR-F can be found in \cite{Harten96}. This section is dedicated to recall some concepts introduced there.

We assume that the function $f$ that we want to manage belongs to a vector space $\cF$ and that, by performing certain method, we will obtain some discrete data, which we will denote by $v$, belonging to some denumerable vector space $V$.
This process will be performed by what is called the \emph{discretization} operator $\cD:\cF\lra V$.
In the example above, $V$ represents the set of images (with a fixed number and distribution of pixels) that can be taken with a digital camera.

\begin{defi}
\label{defi:discretization}
Let $\cF$ and $V$ be vector spaces and $\cD:\cF\lra V$ a linear operator such that $V=\cD(\cF)$. If $V$ has a denumerable basis, say $\{\eta_i\}$, then we say that $\cD$ is a \emph{discretization} operator and we refer to $v:=\cD f$ as the discretization of $f$.
\end{defi}

The vector space $V$ is usually a space of bounded sequences $\ell_\infty(B)$ or of square-summable sequences $\ell_2(B)$, where the indexes are on a discrete domain, for instance $B=\Z^s$. An appropiate norm may be considered on the vector space $\ell_\infty(B)$ so that it has a denumerable Schauder basis, while $\ell_2(B)$ is a Hilbert space. Classical examples of discretization operators are the point-value operator and the cell-average operator.
Let $\cF$ be the set of bounded continuous functions defined on $\R$. Given $X=(x_i)_\iZ$, a grid over $\R$, the associated point-value operator is defined as
$$\cD:\cF \lra \ell_\infty(\Z), \qquad \cD(f)=(f(x_i))_\iZ,$$
and for a given sequence of bounded intervals $C=(c_i)_\iZ$, the associated cell-average operator is defined as
$$\cD:\cF \lra \ell_\infty(\Z), \qquad \cD(f)=\frac{1}{|c_i|}\int_{c_i}f,$$
where $|c_i|=\int_{c_i}1$ is the length of the interval.

The word \emph{multiresolution} in HMR-F comes from the definition of several \emph{resolutions} or \emph{scales}, which can be thought of successive grids in the discretization process. In order to achieve this structure, we need a sequence of discretization operators
fulfilling the \emph{nested} condition (\ref{eq:nested_condition}).
\begin{defi} 
Let $\{\cD_k\}_{k=0}^\infty$ be a sequence of discretization operators
$$\cD_k:\cF\lra V^k, \qquad V^k=\cD_k(\cF)=\spa_i\{\eta_i^k\}.$$
We say that $\{\cD_k\}_{k=0}^\infty$ is \emph{nested} if for all $f\in\cF$ and all $k\geq 0$
\begin{equation}
\cD_{k+1} f=0 \quad \Longrightarrow \quad \cD_{k} f=0. \label{eq:nested_condition}
\end{equation}
\end{defi}

To get a nested sequence of point-value discretizations it is enough to take a sequence of grids $(X^k)_{k\geq0}$ satisfying $X^{k}\subset X^{k+1}$, for instance $X^k=(i2^{-k})_\iZ$. Analogously, the cell-average case demands that each interval of $C^{k}$ is the union of intervals in $C^{k+1}$, such as $c_i^k=[(i-1/2)2^{-k},(i+1/2)2^{-k}]$, $C^k=(c_i^k)_\iZ$, that satisfies $c_{2i}^{k+1} \cup c_{2i+1}^{k+1} = c_i^{k}$.

It is said that the resolution level $k+1$ is finer than $k$ and, equivalently, the space $V^{k}$ is coarser than $V^{k+1}$. Because of (\ref{eq:nested_condition}), these concepts are coherent and allow us to define the \emph{decimation} operators.
\begin{defi}
Let $\{\cD_k\}_{k=0}^\infty$ be a sequence of nested discretization operators.
We define the sequence of \emph{decimation} operators $\{D_{k+1}^{k}\}_{k=0}^\infty$ as
$$D_{k+1}^{k}:V^k\lra V^{k-1}, \qquad v^{k+1}=\cD_{k+1} f \mapsto v^{k}=\cD_{k} f.$$
\end{defi}

A decimation operator projects the data contained at a finer level to a coarser one, without knowledge of the function $f$. In real applications this is essential, since the only available data is discrete.

The \emph{multiresolution setting} is defined as the set of all the spaces $V^k$ and the decimation operators: $(\{V^k\},\{D_{k+1}^{k}\})_{k=0}^\infty$. 
In this framework there are other relevant operators, the \emph{reconstruction} operators, which takes discrete data and tries to approximate the original function. Observe that interpolation techniques can be used as reconstruction operators.
\begin{defi} 
We say that $\cR:V\lra\cF,$ $V=\cD(\cF),$
is a \emph{reconstruction} operator (compatible with $\cD$) if it is a right-inverse of $\cD$, i.e.
\begin{equation} \label{eq:consistency}
\cD\cR=I,
\end{equation}
where $I$ is the identity operator in $V$.
\end{defi}
The condition (\ref{eq:consistency}) is known as \emph{consistency}, and guarantees that $\cR_k$ is injective:
$$ \cR_k v^k = \cR_k w^k \lra v^k = \cD_k \cR_k v^k = \cD_k \cR_k w^k = w^k.$$

The reconstruction operator is crucial in practical applications to manage the data among different $V^k$ (without knowing the true function $f$). On the one hand, $D_{k+1}^{k}$ sends $v_{k+1}:=\cD_{k+1} f$ to $v^k:=\cD_{k} f$ and, on the other hand, we may consider the \emph{prediction} operators, that from $v^{k}$ try to approximate $v^{k+1}$.
\begin{defi} 
We say that $\{P_{k}^{k+1}\}_{k=0}^\infty$ are \emph{prediction} operators for the multiresolution setting, if each one is a right inverse of $D_{k+1}^{k}$ in $V^{k}$,i.e.,
$$P_{k}^{k+1}:V^{k}\lra V^{k+1}, \quad D_{k+1}^{k}P_{k}^{k+1}=I_{k}, \qquad k\geq 0.$$
\end{defi}
Observe that the prediction and the decimation are consistent \eqref{eq:consistency}. As proved in \cite[Theorem 3.2]{Harten96}, there is a useful property that links the prediction and the decimation with the reconstruction and the discretization:
\begin{equation} \label{eq:decimation_prediction_composition}
D_{k+1}^{k}=\cD_{k}\cR_{k+1}, \qquad P_{k}^{k+1}=\cD_{k+1}\cR_{k}.
\end{equation}
The reconstruction and prediction operators may be nonlinear, which is interesting for some applications \cite{ADL11,ADLT06,AL05,ADS16,Aslam18,DL-U18,DL-UM16}. However, $D_{k+1}^{k}$ is always linear, despite its relation with $\cR_k$ according to \eqref{eq:decimation_prediction_composition}.

\section{Construction of multiresolution settings from prediction operators} \label{sec:construction}

In the last section we showed that a multiresolution setting is defined through
the decimation operators, which in turn could be defined from nested discretization operators. From that, prediction operators are taken in a consistent way.
Here we prove that starting from a linear prediction and a reconstruction operators, which fulfill a certain property, we can always find an associated decimation operator. This is done in one of the main results of this section, Theorem \ref{teo:hierar_define}, which can always be applied to linear convergent subdivision schemes. In particular, throughout this section the theory is applied to the (so-called) univariate, uniform, local and binary subdivision schemes, a rather simple class of schemes.

\begin{defi}
\label{defi:SS}
A \emph{subdivision scheme} is a sequence of operators $\{\nS_k^{k+1}:V^k\lra V^{k+1}\}$ such that
\begin{align*}
(\nS_k^{k+1} v)_{i}=\sum_{j} a^k_{i-2j} v_i,
\end{align*}
for some sequences $a^k = \{a^k_i\}_\iZ$ of compact support\footnote{
A function $\phi:A\subset\R\lra\R$ is of \emph{compact support} if $\overline{\{t\in A \, : \, \phi(t)\neq 0\}}$ is a compact subset. A sequence $a$ is of compact support if $\{i\in \Z \, : \, a_i\neq 0\}$ is finite.
}. 
It is convergent if
$$\forall k\geq0\quad  \forall v^k\in V^k \quad \exists \lim_{L \rightarrow \infty} \cR_L P_{L-1}^L  P_{L-2}^{L-1} \cdots P_{k}^{k+1} v^k,$$
where the operator $\cR_k : V^k \rightarrow \cF$ is defined as
$$\cR_k v^k(t) := \sum_\iZ v_i^k \psi(2^kt-i), \qquad \psi(t):= \max\{1-|t|,0\}.$$
\end{defi}

In Definition \ref{defi:SS}, $\cR_k$ is the linear operator that, given $v^k$, constructs the unique piece-wise linear function with nodes at $2^{-k}\Z$ fulfilling $\cR_k (v^k)(i2^{-k}) = v_i^k$.
Hence, a subdivision scheme is convergent if, and only if, the sequence of piece-wise linear functions converges to a function.
A classical result in subdivision theory is that $\psi$ can be replaced by other continuous compactly supported function (see for example \cite[Lemma 2.2]{Dyn92}), although working with piece-wise linear functions is usually preferable as they are easier to visualize.

Subdivision schemes are usually defined on $V^k=\ell_\infty(\Z)$ and they converge in $\cF=\cC(\R) \cap L_\infty(\R)$, but, in practice, the data set is always finite, i.e. $V^k=\R^{n_k}$. Some of the following require $V^k$ and $\cF$ to be Hilbert spaces, so henceforward we suppose that
\begin{align*}
 V^k &= \ell_{2}(\Z):= \{ v^k \, : \, \sum_\iZ |v_i^k|^2 < +\infty \},\\
  \cF &= L_{2}(\R):= \{ f \, : \, \int_\R |f(t)|^2 < +\infty \}.
 \end{align*}
This choice of spaces trivially includes the finite case, since we can suppose that the sequence is identically zero outside the range. Recall that their inner products are
\begin{align*}
\langle u^k,v^k \rangle \mapsto \sum_\iZ u_i^k v_i^k , \qquad
\langle f,g \rangle \mapsto \int_\R f(t) g(t).
 \end{align*}

We can extend Definition \ref{defi:SS} to HMR-F. The subdivision operators are clearly prediction operators, while the piece-wise linear function is a reconstruction operator.

\begin{defi}
\label{defi:hierar_pred}
Let $\{P_{k}^{k+1}\}_{k=0}^\infty$ be a sequence of injective operators such that
$$ P_{k}^{k+1}:V^{k} \lra V^{k+1}, \qquad V^k=\espan\{\eta^k_i\}.$$
We say that  $\{P_{k}^{k+1}\}_{k=0}^\infty$ is a \emph{convergent} sequence of prediction operators in $\cF$ if there exists a sequence of injective operators $\{\cR_k\}_{k=0}^\infty$, $\cR_k:V^k\lra\cF$, such that $\cR_{k}(V^{k})\subset \cR_{k+1}(V^{k+1})$ and
$$\forall k\geq 0 \quad \forall v^k\in V^k \quad \exists \lim_{L\rightarrow\infty} \cR_L P_{L-1}^L P_{L-2}^{L-1} \cdots P_k^{k+1} v^k.$$
\end{defi}

We remark that the operator in Definition \ref{defi:SS} fulfills the condition $\cR_{k}(V^{k})\subset \cR_{k+1}(V^{k+1})$ because $\psi(t/2) = \frac12 \psi(t-1) + \psi(t) + \frac12 \psi(t+1)$. In addition, $\cR_k$ can be defined using other compactly supported $\psi$ satisfying $\psi(t/2) = \sum b_i \psi(t-i)$, for some compactly supported sequence $b=(b_i)$, which is a classic result in subdivision theory. Indeed,
\begin{align} \label{eq:Rk_in_Rk1}
\begin{split}
\cR_k v^k(t) &= \sum_\iZ v_i^k \psi(2^kt-i) = \sum_\iZ v_i^k \psi((2^{k+1}t-2i)/2) \\
&= \sum_\iZ v_i^k \sum_{j\in\Z} b_j \psi(2^{k+1}t-2i-j)\\
&= \sum_{j\in\Z} \left (\sum_\iZ v_i^k b_{j-2i}\right ) \psi(2^{k+1}t-j)= \cR_{k+1} \tilde v^k (t),
\end{split}
\end{align}
where $\tilde v^k_j :=  \sum_\iZ v_i^k b_{j-2i}$.

Now that we have chosen prediction and reconstruction operators from the framework of subdivision schemes, we set to define discretization operators $\cD_k$ that are consistent with them. In the definitions above, applying the reconstruction operator to our data gives us a function, $\cR_k v^k$, that should approximate $f$. This motivates us to define $\cD_k$ so that $\cR_k v^k$ is the best possible approximation of $f$ that $\cR_k$ can reach. The next result, which is a direct application of the Hilbert projection theorem to our setting, formalizes this idea.

\begin{teo} \label{teo:cSk+}
Let $\cF$ be a Hilbert space and let $\dS_k:V^k\lra  \cF$ be an injective linear operator. Then the operator defined by
$$\cD_k f:=\argmin_{v\in V^k} \|f-\dS_k v\|^2$$
is well defined, linear and it is a left inverse of $\dS_k$, that is, $\cD_k\dS_k = I_{V_k}$.
\end{teo}
\begin{proof}
Note that
$\cV_k:=\dS_k(V^k)\subset\cF$ is a subspace of $\cF$ because $\dS_k$ is linear. Given $f\in\cF$, using the orthogonal projection there exist two unique vectors $f_1 \in \cV_k,f_2 \in \cV_k^\perp$ such that $f=f_1+f_2$, where $f_1 = \mathbb{P}_{\cV_k}(f) = \argmin_{g\in \cV_k} \|f-g\|^2$, $\mathbb{P}_{\cV_k}$ being the orthogonal projection onto the set $\cV_k$. Since $\dS_k$ is injective, there exists one and only one $v_1\in V^k$ such that $ f_1=\dS_k(v_1)$, so consequently $v_1=\argmin_{v\in V^k} \|f-\dS_k v\|^2$. Then $\cD_k$ is defined as the composition $\left ( \left . \dS_k \right | _{\cV_k} \right )^{-1} \circ \mathbb{P}_{\cV_k}$ and therefore the linearity follows directly form the linearity of both $\dS_k$ and $\mathbb{P}_{\cV_k}$. Finally, since obviously $\mathbb{P}_{\cV_k} \circ \dS_k=\dS_k$, $\left ( \left . \dS_k \right | _{\cV_k} \right )^{-1} \circ \mathbb{P}_{\cV_k} \circ \dS_k = \left ( \left . \dS_k \right | _{\cV_k} \right )^{-1}  \circ \dS_k = I_{\cV_k}$.
\end{proof}

\noindent From the last proof we deduce the following useful equality
$$ f= \dS_k \cD_k f + f_2 \in \cV_k\oplus\cV_k^\perp.$$


Now that we got a sequence of discretization operators from $\{\cR_k\}_{k\geq0}$, we prove that it is nested.
\begin{cor} \label{cor:nested}
Let $\cF$ be a Hilbert space, $\dS_k:V^k\lra  \cF$ injective linear operators satisfying $\cR_{k}(V^{k})\subset \cR_{k+1}(V^{k+1})$, and $\cD_k :\cF \lra V^k$ defined as $\cD_k f:=\argmin_{v\in V^k} \|f-\dS_k v\|^2$. Then
$\{\cD_k\}$ is a nested sequence of discretizations.
\end{cor}
\begin{proof}
We have to prove that $\cD_{k+1} f = 0 \Rightarrow \cD_{k} f = 0.$ From
$\cD_{k+1} f = 0$ we deduce that $ f = 0 + f_2 \in \cV_{k+1}\oplus\cV_{k+1}^\perp$, thus $f\in\cV_{k+1}^\perp.$
But by hypothesis $\cV_k \subset \cV_{k+1}$, which implies $ \cV_{k+1}^\perp \subset \cV_{k}^\perp$.
As a consequence $ f = 0 + f_2 \in \cV_{k}\oplus\cV_{k}^\perp$, then $\cD_{k} f =0.$
\end{proof}
From now on we will assume that $P_k^{k+1}$ is linear, so the following results can only be applied to linear subdivision schemes. Nevertheless, the nonlinear schemes developed in the literature \cite{ADL11,ADLT06,AL05,ADS16,Aslam18,DL-U18,DL-UM16} are usually designed in the point-value or in the cell-average framework, so they do not need the theory presented here.

Now, in order to get a similar definition for the decimation operators, we apply Theorem \ref{teo:cSk+} considering $\cF=V^{k+1}$ and prediction operators instead of reconstruction operators.
\begin{cor} \label{cor:decimation}
Let $V^{k},V^{k+1}$ be Hilbert spaces and let $\nS_{k}^{k+1}:V^{k}\lra V^{k+1}$ be an injective linear operator. Then the operator defined by
$$D_{k+1}^{k} v^{k+1}:=\argmin_{v^{k}\in V^{k}} \|v^{k+1}-\nS_{k}^{k+1} v^{k}\|^2$$
is well defined, linear and it is a left inverse of $\nS_{k}^{k+1}$: $D_{k+1}^{k} \nS_{k}^{k+1} = I_{V_k}$.
\end{cor}
Note that both in Theorem \ref{teo:cSk+} and in Corllary \ref{cor:decimation} different discretization and decimation operators are obtained if the inner product is changed. To prove Theorem \ref{teo:hierar_define}, which deals with the consistency of these new operators, the next inner product may be considered.

\begin{lem} \label{lem:inner}
Let $\cF$ be a Hilbert space with inner product $\langle \cdot,\cdot\rangle $ and let $V_k$ be another Hilbert space. Let $\cR_k:V_k \rightarrow \cF$ be a linear injective operator, and denote for any $v,w\in V^k$
$$\langle v,w\rangle _k:=\langle \cR_k v,\cR_k w\rangle .$$
Then $\langle \cdot,\cdot\rangle _k$ is an inner product of $V^k$.
\end{lem}

\noindent The proof is straightforward and we do not include it. We define $ \|v\|_k:=\sqrt{\langle v,v\rangle _k} = \|\dS_{k}v\|$, $ \forall v\in V^k.$

Section 4.D. of \cite{Harten96}, and in particular the Theorem 4.5, states that for any convergent sequence of prediction operators $\{P_{k}^{k+1}\}$ there exists a sequence of reconstruction operators $\{\cR_k^H\}$ which are still consistent with $\{\cD_k\}$, fulfilling $P_k^{k+1} = \cD_{k+1}\cR_k^H= \cD_{k+1}\cR_k$ and
\begin{equation} \label{eq:hier}
\cR_{k+1}^H P_{k}^{k+1} = \cR_{k}^H.
\end{equation}
This is an analogous result to the following well-known fact in subdivision theory (see \cite[Theorem 2.4]{Dyn92}): For any convergent subdivision scheme, there exists a compactly supported function $\phi$ such that
\begin{equation} \label{eq:decimation_eq}
\phi(t)=\sum_i a_i \phi(2t-i).
\end{equation}
Indeed, $\cR_k^H$ is obtained if $\psi=\phi$, since \eqref{eq:Rk_in_Rk1} with $b=a$ implies $\tilde v^k = P_k^{k+1} v^k$, hence \eqref{eq:hier}.

\begin{teo} \label{teo:hierar_define}
Let $\{P_{k}^{k+1}\}$ be a convergent sequence of prediction operators. Let us denote $\|v^k\|_{k,H} := \|\cR_k^H v^k\|$, the derived norm of Lemma \ref{lem:inner} applied to $\cR_k^H$. Then
$$\cD_k f:=\argmin_{v\in V^k} \|f-\dS_k^H v\|^2,
\qquad D_{k+1}^{k} w = \argmin_{v\in V^{k}} \|w-P_{k}^{k+1} v\|^2_{k+1,H}$$
are consistent discretization and decimation operators of $\cR_k^H$ and $P_{k}^{k+1}$, respectively. Furthermore the discretization operators are nested and
$$
D_{k+1}^{k}=\cD_{k}\cR_{k+1}^H, \qquad P_{k}^{k+1}=\cD_{k+1}\cR_{k}^H.
$$
\end{teo}

\begin{proof}
By hypothesis
$$ \cD_{k+1} \dS_{k}^H w = \argmin_{v\in V^{k+1}} \|\dS_{k}^H w-\dS_{k+1} ^H v\|^2.$$
Choosing $v=\nS_{k}^{k+1} w$ and using \eqref{eq:hier},
$$\dS_{k}^H  w-\dS_{k+1}^H v = \dS_{k}^H  w-\dS_{k+1}^H \nS_{k}^{k+1} w = \dS_{k}^H  w-\dS_{k}^H  w=0 .$$
As a consequence, $ \cD_{k+1} \dS_{k}^H  = P_{k}^{k+1}.$
On the other hand, using \eqref{eq:hier} and Lemma \ref{lem:inner} for $\cR_{k+1}^H$:
\begin{align*}
\cD_{k}\cR_{k+1}^H w&= \argmin_{v\in V^{k}} \|\dS_{k+1}^H  w-\dS_{k+1}^H  \nS_{k}^{k+1} v\|^2 = \argmin_{v\in V^{k}} \|\dS_{k+1}^H ( w-\nS_{k}^{k+1} v)\|^2 \\
&= \argmin_{v\in V^{k}} \| w-\nS_{k}^{k+1} v\|^2_{k+1,H} = D_{k+1}^{k} w.
\end{align*}
\end{proof}

If we start by choosing the discretization operators when constructing a multiresolution framework, a consistent reconstruction must be selected afterwards, and this choice is not unique. For instance, in the point-value framework, the reconstruction can be any interpolation technique. The next result shows how this fact is translated into our point of view, since we prove that, when starting from the reconstruction technique, any consistent discretization operator can be obtained
as in Theorem \ref{teo:hierar_define} by choosing a suitable inner product for $\cF$.

\begin{teo} \label{teo:equivalence}
Let $\cD_k, \cR_k$ be a consistent pair of discretization and reconstruction linear operators.
Then, there exists a scalar product $\langle\langle\cdot,\cdot\rangle\rangle$ of $\cF$ such that
$$\cD_kf=\argmin_{v\in V^k} |||f-\cR_k v|||^2,$$
where $|||f|||^2:=\langle\langle f,f\rangle\rangle$.
\end{teo}
\begin{proof}
Let us consider $ \cW_k f=f-\cR_k\cD_k f$, so the next decomposition is obtained
$$ f= \cW_k  f+ \cR_k\cD_k f,$$
and an associated inner product
$$\langle\langle f,g\rangle\rangle:=\langle \cW_k  f,\cW_k  g\rangle  + \langle \cR_k\cD_k f,\cR_k\cD_k g\rangle .$$
Indeed, it is an inner product because of the linearity of the involved operators and the properties of the inner product $\langle \cdot,\cdot\rangle $. This is easy to check, so we will only prove one of the properties as an example:
\begin{align*}
\langle\langle f,f\rangle\rangle= 0 &\rightarrow \langle \cW_k  f,\cW_k  f\rangle  =0 \quad \wedge \quad \langle \cR_k \cD_k f,\cR_k\cD_k f\rangle =0 \\ 
&\rightarrow \cW_k  f=0 \quad \wedge \quad \cR_k\cD_k f=0 \rightarrow f = \cW_k  f + \cR_k\cD_k f = 0.
\end{align*}
For the new inner product, $\cW_k(\cF)$ and $\cR_k\cD_k(\cF)$ are orthogonal, as we prove in what follows. First, note that
$$ \cW_k\cR_k= \cR_k-\cR_k\cD_k\cR_k = \cR_k-\cR_k =0,\quad \cD_k\cW_k=\cD_k - \cD_k\cR_k\cD_k =\cD_k-\cD_k=0.$$
Then
\begin{align*}
\langle\langle\cW_k f, \cR_k\cD_k g\rangle\rangle&= 
\langle \cW_k \cW_k f,\cW_k  \cR_k\cD_k g\rangle  + \langle \cR_k\cD_k \cW_k f,\cR_k\cD_k \cR_k\cD_k g\rangle\\
&= \langle \cW_k \cW_k f,0\rangle  + \langle 0,\cR_k\cD_k \cR_k\cD_k g\rangle =0.
\end{align*}
Hence for $|||f|||:=\sqrt{\langle\langle f,f\rangle\rangle}$, we have that
\begin{align*}
|||f-\cR_k v|||^2 &= |||\cW_k  f+ \cR_k\cD_k f -\cR_k v|||^2\\
&= |||\cW_k f + \cR_k(\cD_k f - v)|||^2 =|||\cW_k f|||^2 + |||\cR_k(\cD_k f - v)|||^2,
\end{align*}
and as a consequence
$$\cD_kf=\argmin_{v\in V^k} |||f-\cR_k v|||^2.$$
\end{proof}

\section{Computation of the decimation and discretization operators in a practical situation} \label{sec:example}

In this section we apply the Theorem \ref{teo:hierar_define} to a subdivision scheme of the form specified in Definition \ref{defi:SS} . That is $(P_k^{k+1}v)_i=\sum_{j\in\Z} a_{i-2j} v_j,$ being $a\in\ell_{2}(\Z)$ compactly supported. 

Let us consider $V^k=\ell_{2}(\Z)$. If the scheme converges, then exists a compactly supported function $\phi$ satisfying \eqref{eq:decimation_eq}. This implies that $\cR_k^H v := \sum_\iZ v_i \phi(2^kt-i)$, because
\begin{align*}
\left (\cR_{k}^H v\right )(t) &= \sum_{j\in\Z} v_j \phi(2^{k} t-j) = \sum_{j\in\Z} v_j \sum_\iZ a_i \phi(2^{k+1} t -2j-i) \\
&= \sum_{j\in\Z} v_j \sum_\iZ a_{i-2j} \phi(2^{k+1} t -i) = \sum_\iZ  \phi(2^{k+1} t -i) \sum_{j\in\Z} a_{i-2j} v_j\\
&=\sum_\iZ (P_k^{k+1} v)_i \phi(2^{k+1}t-i) = \left (\cR_{k+1}^H P_k^{k+1} v\right )(t).
\end{align*} 

Note that all the sums that appear in the calculations above are actually finite sums given the compact support of both $a$ and $\phi$, so no arguments about convergence are needed.
Now, let us compute the expression of the discretization operator given in Theorem \ref{teo:hierar_define}, which is consistent with $\cR_k^H$:
$$\cD_k f:=\argmin_{v\in V^k} \|f-\dS^H_k v\|^2, \qquad f\in L_2(\R).$$
By Theorem \ref{teo:cSk+}, we know that $\cD_k f=u$ with
$$f = \dS^H_k u + f_2 \in \dS^H_k(\ldZ) \oplus \dS^H_k(\ldZ)^\perp.$$
Then, we have to find $u\in\ldZ$ such that
$$ \langle f - \dS^H_k u , \dS^H_kv\rangle =0 \quad \forall v\in\ldZ.$$
Since $ \sum_i \phi(t-i) = 1$, we deduce that $\int_\R \phi = 1$, which is a classic result on subdivision theory.
Then
\begin{align*}
\langle f&-\dS^H_k u,\dS^H_kv\rangle  = \int_\R (f(t)-\sum_\iZ u_i \phi(2^kt-i))(\sum_{j\in\Z} v_j \phi(2^kt-j)) dt\\
  &= \sum_{j\in\Z} v_j \int_\R  \phi(2^kt-j)f(t)dt - \sum_{i,j\in\Z} u_i v_j \int_\R   \phi(2^kt-i) \phi(2^kt-j) dt\\
  &= \sum_{j\in\Z} v_j \int_\R  \phi(2^kt-j)f(t)dt - \sum_{i,j\in\Z} u_i v_j 2^{-k}\eta_{j-i}
\end{align*}
where $\eta_i := \int_\R \phi(t) \phi(t-i) dt$. Again, $\phi$ being compactly supported makes all the sums that appear above finite sums, and moreover $\eta_{j-i}=0$ for $|j-i|> M$, $M \in \Z_+$ large enough. Therefore
\begin{align*}
\langle f&-\dS^H_k u,\dS^H_kv\rangle = \sum_{j\in\Z} v_j \left( \int_\R  \phi(2^kt-j)f(t)dt - 2^{-k}\sum_{i=-A}^{A} u_{j-i}  \eta_{i}\right )
\end{align*}
The right-hand side of this equality is 0 for all $v\in\ldZ$ if and only if there exists $u\in \ell_2(\Z)$ such that
$$ \sum_{i=-A}^{A}  \eta_{i} u_{j-i} = 2^{k}\int_\R  \phi(2^kt-j)f(t)dt,$$
which is an infinite system of linear equations, whose matrix $E=(\eta_{i-j})_{i,j\in\Z}$ is Toeplitz. Theorem \ref{teo:cSk+} assures that the system has a unique solution because it guarantees the existence of $\cD_k$, but to find an explicit expression is not always possible, so numerical algorithms specialized in Toeplitz systems may be needed.

To illustrate this process with an easy example, let us consider the subdivision scheme $(P_k^{k+1}v)_{2i}=(P_k^{k+1}v)_{2i+1}=v_i$, which can also be written as $(P_k^{k+1}v)_i=\sum_j a_{i-2j} v_i$ with $ a_{1}=a_{0}=1$ and $a_{i}=0$ for all $i \not\in\{0,1\}$, where $\phi$ is just the box function $\phi(t)=\chi_{[0,1)}(t)$. We will call this scheme the Step Scheme, which is closely related with the Haar wavelet. Now $\eta_0=1$ and $\eta_i = 0$ for $i\neq 0$, and the system of equations becomes diagonal:
$$ u_{j} = 2^{k}\int_\R  \phi(2^kt-j)f(t)dt = 2^{k}\int_{2^{-k}j}^{2^{-k}(j+1)} f(t)dt.$$
Observe that, in this particular case, $(\cD_k f)_i = 2^{k}\int_{2^{-k}i}^{2^{-k}(i+1)} f(t)dt$, which is the cell-average discretization.

Now, to get the decimation operators which are consistent with our choice of a subdivision scheme, we compute the inner product of Theorem \ref{teo:hierar_define}:
\begin{align*}
\langle v,w\rangle_{k,H}&=\langle \dS^H_{k}v,\dS^H_{k} w\rangle=\langle \sum_\iZ v_i \phi(2^kt-i),\sum_{j\in\Z} w_j \phi(2^kt-j)\rangle  \\
&= \sum_{i,j\in\Z} v_i w_j \langle \phi(2^kt-i),\phi(2^kt-j)\rangle  = 2^{-k}\sum_{i\in\Z} v_i \sum_{j=-A}^A w_{i-j} \eta_j \\
&= 2^{-k} v^T E w,
\end{align*}
where, in the last line, $v$ and $w$ are considered column vectors, denoting by $v^T$ its transpose.
For the Step Scheme, $\langle v,w\rangle_{k,H} = 2^{-k}\sum_{i\in\Z} v_i w_{i}$, and we will see that the associated decimation operator is just the one corresponding to the cell-average framework.
Recall from Theorem \ref{teo:hierar_define} the definition of the decimation operator:
\begin{align*}
D_{k+1}^{k} w = \argmin_{v\in V^{k}} \|w-\nS_k^{k+1} v\|^2_{k+1,H}.
\end{align*}
As above, we know by Theorem \ref{teo:cSk+} that $D_{k+1}^{k} w = u$ with $w-\nS_k^{k+1} u \in \nS_k^{k+1}(\ldZ)^\perp$,
but now the orthogonality is with respect to the product $\langle \cdot,\cdot\rangle _{k+1,H}$, so now our equation is
$$ \langle  w-\nS_k^{k+1} u , P_k^{k+1} v \rangle _{k+1,H} = 0, \qquad \forall v\in\ldZ.$$
This expression can be straightforwardly developed: for all $v\in\ell_2(\Z)$,
\begin{align*}
0 &= \langle  w-\nS_k^{k+1} u , P_k^{k+1} v \rangle _{k+1,H} = 2^{-k-1} ( P_k^{k+1} v)^T E ( w-\nS_k^{k+1} u)  \\
&= 2^{-k-1} v^T ( P_k^{k+1} )^T E ( w-\nS_k^{k+1} u)
\end{align*}
if, and only if, $( P_k^{k+1} )^T E \nS_k^{k+1} u = ( P_k^{k+1} )^T E w$.
This is again a Toeplitz system, as we prove in the following by checking that
\begin{equation} \label{eq:toeplitz1}
\left ( ( P_k^{k+1} )^T E \nS_k^{k+1}\right )_{m,n} = \left ( ( P_k^{k+1} )^T E \nS_k^{k+1}\right )_{m-n,0}, \qquad \forall m,n\in\Z.
\end{equation}
Denoting by $\delta^n$ the Kronecker delta,
\begin{align*}
&\left ( ( P_k^{k+1} )^T E \nS_k^{k+1}\right )_{m,n} = (\delta^m)^T ( P_k^{k+1} )^T E \nS_k^{k+1} \delta^n =   ( P_k^{k+1} \delta^m )^T E (\nS_k^{k+1} \delta^n),
\end{align*}
but $(\nS_k^{k+1} \delta^n)_i = \sum_{j\in\Z} a_{i-2j}\delta^n_j = a_{i-2n}$, so
\begin{align*}
&\left ( ( P_k^{k+1} )^T E \nS_k^{k+1}\right )_{m,n} =  ( (a_{i-2m})_{i\in\Z} )^T E (a_{i-2n})_{i\in\Z} = \sum_{i,j\in\Z} a_{i-2m} \eta_{i-j} a_{j-2n}.
\end{align*}
Taking the correct values of $i,j$, for the right side of \eqref{eq:toeplitz1} we obtain
\begin{align*}
&\left ( ( P_k^{k+1} )^T E \nS_k^{k+1}\right )_{m-n,0}  = \sum_{i,j\in\Z} a_{i-2m+2n} \eta_{i-j} a_{j}.
\end{align*}
With a simple change of summation variables, now we can see that \eqref{eq:toeplitz1} holds true.

For the Step Scheme, this Toeplitz matrix is a diagonal matrix with constant diagonal 2:
\begin{align*}
\left (( ( P_k^{k+1} )^T E \nS_k^{k+1}\right )_{m,n} &= \sum_{i,j\in\Z} a_{i-2m+2n} \eta_{i-j} a_{j} \\
&= \sum_{i\in\Z} a_{i-2m+2n} a_{i} = a_{2(n-m)} + a_{1+2(n-m)},
\end{align*}
which is 2 if $n=m$ and 0 if $n\neq m$. Hence $2 u = ( P_k^{k+1} )^T E w$. The right part can be calculated analogously:
\begin{align*}
2 u_n &= ( P_k^{k+1} \delta^n )^T E w = (a_{i-2n})_{i\in\Z} )^T E w = \sum_{i,j\in\Z} a_{i-2n} \eta_{i-j} w_j \\
&= \sum_{j\in\Z} \eta_{2n-j}w_j + \eta_{2n+1-j}w_j = w_{2n}+w_{2n+1}.
\end{align*}
Finally we have arrived at our desired conclusion since we have obtained that
$$ (D_{k}^{k-1} w)_i = u_i = \frac12(w_{2i}+w_{2i+1}),$$
which is the usual decimation in the cell-average framework.

\bibliographystyle{plain}
\bibliography{biblio}

\end{document}